\documentclass[11pt]{amsart}

\usepackage[english]{babel}
\usepackage{pgfplots}
\pgfplotsset{compat=1.17}
\usepackage[a4paper, margin=2cm]{geometry}
\usepackage{graphicx}
\usepackage{amsmath, amsthm, amssymb}
\usepackage{enumitem}
\setlist[1,enumerate]{label={(\roman*)}}
\setlist[1]{leftmargin=1.5em}
\setlist[2,enumerate]{label={(\alph*)}}
\setlist{itemsep=0pt, topsep=\smallskipamount, listparindent=1em}
\usepackage{tikz-cd}
\usepackage{hyperref}
\hypersetup{colorlinks=true}

\def\N{\mathbb N}

\def\R{\mathbb R}
\def\C{\mathbb C}

\theoremstyle{plain}
\newtheorem{theorem}{Theorem}

\newtheorem{lemma}[theorem]{Lemma}

\theoremstyle{definition}
\newtheorem{definition}[theorem]{Definition}

\theoremstyle{remark}
\newtheorem{remark}[theorem]{Remark}
\newtheorem{example}[theorem]{Example}

\begin{document}
\title{Hermitian Distance Degree of Unitary-Invariant Matrix Varieties}
\author{Nikhil Ken}

\email{Nikhil.ken@unifi.it}

\begin{abstract}
We study the Hermitian distance degree, a real enumerative invariant counting critical points of the squared Hermitian distance function, for matrix varieties invariant under left and right unitary actions. For such a variety \(M \subset \mathbb{C}^{n\times t}\), we prove that its Hermitian distance degree equals the real Euclidean distance degree of the associated absolutely symmetric variety of singular values. Equivalently, for a generic data matrix, Hermitian distance critical points on \(M\) are obtained by lifting Euclidean distance critical points from the singular-value slice. We also establish a Hermitian slicing theorem, paralleling the Bik--Draisma principle, which reduces the critical point count to a diagonal slice. As a motivating example, we recover a geometric Hermitian analogue of the Eckart-Young theorem.
\end{abstract}

\maketitle

\section{Introduction}
In this article, 
we study a real enumerative invariant attached to Hermitian distance critical points,
which we call the Hermitian distance degree (introduced and studied in detail in \cite{Furchi2025_HermitianDistanceDegree}),
focusing on matrix varieties invariant under left and right unitary actions. Recall that the ED degree of an algebraic variety
$X \subset \mathbb{C}^n$ is the number of critical points in $X^{reg}$ of the squared distance function
\[
d_u(x) = \|x - u\|^2,
\]
arising from a Euclidean inner product and a general data point $u \in \mathbb{C}^n$; see \cite{draisma2016euclidean}, \cite{drusvyatskiy2017euclidean}, \cite{HorobetWeinstein2019_OffsetPersistentHomology}, \cite{Sodomaco2021_ProductSingularValues} and \cite{baaijens2015euclidean}.

In the Hermitian setting, for matrices $A,B \in \mathbb{C}^{n \times t}$, we consider the Hermitian inner product
\[
\langle A,B \rangle = \mathrm{Tr}(A B^*),
\]
where $B^*$ is the conjugate transpose of $B$. The \emph{Hermitian distance degree} is a real enumerative invariant attached to the Hermitian
distance critical points.  Concretely, we view $M\subseteq\C^{n\times t}=V$ as a real algebraic set
in $V^{\R}\cong\R^{2nt}$ and use the real inner product $q=\mathrm{Re}\langle\cdot,\cdot\rangle$. Since we work in $V^{\mathbb R}$ with $q=\mathrm{Re}\langle\cdot,\cdot\rangle$, we count real critical points; complex critical points of the complexified equations are not critical points of the real problem.

For a data matrix $Y\in\C^{n\times t}$, let $d_Y(X)=\|X-Y\|^2$ be the squared distance function
with respect to $q$.  For $Y$ outside a certain semialgebraic discriminant locus, the number of
\emph{real} critical points of $d_Y$ on $M^{\mathrm{reg}}$ is finite and constant on the chamber
of $Y$; we denote this chamber-wise constant by $\mathrm{HDdeg}(M;Y)$. The set of all such chamberwise values is denoted by
\[
\mathrm{HDdeg}(M):=\{\mathrm{HDdeg}(M;Y): Y\in V^{\R}\setminus \Delta_M\}\subset\N.
\]

Let $f: \mathbb{C}^{n\times t} \to \mathbb{R}$ be defined by
\[
f(x) = \|x - y\|^2 = \langle x - y, x - y \rangle,
\]
where $\langle \cdot, \cdot \rangle$ is the Hermitian inner product.
For $f(x)=\|x-y\|^2$ induced by the Hermitian product, the differential satisfies
\[
df_x(v)=\langle v,x-y\rangle+\langle x-y,v\rangle
=2\,\mathrm{Re}\langle x-y,v\rangle .
\]
This observation means that the Hermitian distance problem is governed by the real symmetric bilinear form
\[
(\cdot,\cdot)_{\mathbb{R}} \;:=\; \mathrm{Re}\langle \cdot,\cdot\rangle
\]
on the real vector space \(V^{\mathbb{R}}\). In particular, Hermitian distance critical points are precisely Euclidean distance
critical points for \(X\), viewed as a real algebraic set in \(V^{\mathbb{R}}\) with respect to \((\cdot,\cdot)_{\mathbb{R}}\),
and we count only the \emph{real} critical points of this real optimization problem; hence the relevant enumerative invariant is a
\emph{real} Euclidean distance degree, denoted \(\mathrm{\mathbb{R}EDdeg}\).

In \cite{drusvyatskiy2017euclidean}, the authors study the ED degree of matrix varieties invariant under the action of
orthogonal groups, and show that it coincides with the ED degree of a variety determined by the singular values. This idea was
further generalized in \cite{bik2018eddegrees}, where it was shown that the ED degree of a variety $X$ with a suitable group action
can sometimes be computed by slicing $X$ with an appropriate subspace.

Building on these ideas, we show in Theorem~\ref{main} that the HD degree of a Unitary-invariant matrix variety
$M \subset \mathbb{C}^{n \times t}$ equals the real ED degree of the associated absolutely symmetric variety
$S \subset \mathbb{R}^n$ of singular values. Concretely, for a generic data matrix $Y=U\mathrm{diag}(y)V^*$
, one has
\[
\mathrm{HDdeg}(M;Y)=\mathrm{\mathbb{R}EDdeg}(S;y).
\]
Moreover, this common value is constant on chambers of $\mathbb{R}^n$; when the chamber is clear from context,
we denote it simply by $\mathrm{HDdeg}(M)$ (or equivalently $\mathrm{\mathbb{R}EDdeg}(S)$).

Beyond the numerical equality, Theorem~\ref{main} gives a concrete description of the critical points themselves.  
For a generic data matrix $Y = U\mathrm{diag}(y)V^*$, every Hermitian distance critical point on $M$ is obtained by lifting a real Euclidean distance critical point of $y$ on $S$, using the same singular vectors $U,V$.  
Since the counting problem is genuinely real, the number of critical points can depend on the chamber of $y$, and may jump when $y$ crosses the ED discriminant. We illustrate this behaviour explicitly in Example \ref{paras}, \ref{ex:det-mag-2x2}.

We also prove a Hermitian analogue of the slicing principle of Bik--Draisma by passing to \(V^{\mathbb{R}}\) and exploiting that
unitary representations are orthogonal with respect to \(q=\mathrm{Re}\langle\cdot,\cdot\rangle\).

As a motivating case, we revisit the Eckart-Young theorem (Theorem~\ref{gg}) geometrically for Hermitian distance, recovering the
critical points of the Hermitian distance function to the rank-$k$ determinantal variety (compare \cite{ottaviani_paoletti_svd} for
the Euclidean setting). This becomes a special case of Theorem~\ref{main}. Hermitian distance critical equations also arise in the study of closest product states
and geometric measures of entanglement; see \cite{HillingSudbery2010}.

\subsection* {Conventions}
Throughout we identify the complex vector space $V=\C^{n\times t}$ as a real vector space
\[
V^\R \;\cong\; \R^{2nt},
\]
and we equip $V^\R$ with the real Euclidean inner product
\[
q(A,B):=\mathrm{Re}\langle A,B\rangle \;=\; \mathrm{Re}\big(\mathrm{Tr}(AB^*)\big).
\]
All algebraic subsets of $V$ will be viewed as real algebraic sets in $V^\R$ when counting
(real) distance critical points.
For $n\le t$, we write $\mathrm{diag}(d)$ for the rectangular diagonal matrix in $\C^{n\times t}$
with $(\mathrm{diag}(d))_{ii}=d_i$ for $1\le i\le n$ and all other entries zero.

\medskip
\noindent\textbf{Ambient-space subscripts.}
When we write $\R\mathrm{EDdeg}$ or $\mathrm{HDdeg}$ with a subscript, the subscript indicates
the \emph{ambient real inner-product space} in which the critical points are computed.
When the ambient space is clear from context, we omit the subscript.

\section*{Acknowledgements}
This work was supported by the European Union’s Horizon Europe programme under the Marie Sk\l{}odowska--Curie Actions,
HORIZON--MSCA-2023-DN-JD (grant agreement No.~101120296, TENORS). I thank Giorgio Maria Ottaviani for comments on multiple drafts,
and Davide Furch\`{\i} for helpful discussions on the Hermitian distance degree.

\section{Preliminaries} 

\begin{theorem} \cite{golub2013matrix}(2.4.4)
    Let $ A $ be a matrix in $ \mathbb{C}^{m \times n} $. Then $ A $ can be factored as

$$
A = U \Sigma V^*
$$

where $ U \in \mathbb{C}^{m \times m} $ is unitary, $ V \in \mathbb{C}^{n \times n} $ is unitary, and $ \Sigma \in \mathbb{R}^{m \times n} $ has the form

$$\Sigma = \mathrm{diag}(s_1, s_2, \dots, s_p),
$$
where $ p = \min(m, n) $. The $s_i's$ are called the singular values of $A$, and they are the square roots of the eigenvalues of $AA^*$
\end{theorem}
\begin{definition}
Let $S \subseteq \mathbb{R}^n$ be a real algebraic set and let $y \in \mathbb{R}^n$.
A smooth point $x \in S^{\mathrm{reg}}$ is an \emph{ED critical point} of $y$ on $S$ if
\[
x-y \perp T_x S,
\qquad\text{i.e.}\qquad (x-y)^{\top}a = 0 \ \ \text{for all } a \in T_xS.
\]
Analogously, let $M \subseteq \mathbb{C}^{n\times t}$ be a matrix variety and $Y \in \mathbb{C}^{n\times t}$.
A smooth point $X \in M^{\mathrm{reg}}$ is an \emph{HD critical point} of $Y$ on $M$ if
\[
\mathrm{Re}\,\mathrm{Tr}\!\big((Y-X)A^{*}\big)=0
\qquad\text{for all } A \in T_XM.
\]
\end{definition}

\begin{lemma}\label{a}
   If $A_1 = u_1 v_1^*$ and $A_2 = u_2 v_2^*$ are rank-one matrices, then the Hermitian inner product given by $\langle A_1, A_2 \rangle = \mathrm{Tr}(A_1 A_2^*)$ is also given by

$$
\langle A_1, A_2 \rangle = \langle u_1, u_2 \rangle \langle v_1, v_2 \rangle.
$$\end{lemma}
\begin{proof}
  $$  \mathrm{Tr}\left[
\begin{pmatrix}
u_{11} \\
\vdots \\
u_{m1}
\end{pmatrix}
\begin{pmatrix}
v_{11} & \cdots & v_{1n}
\end{pmatrix}
\cdot
\begin{pmatrix}
\overline{v_{21}} \\
\vdots \\
\overline{v_{2n}}
\end{pmatrix}
\begin{pmatrix}
\overline{u_{21}} & \cdots & \overline{u_{2m}}
\end{pmatrix}
\right]$$ 
So we have $\langle A_1, A_2 \rangle =\sum_{i} u_{1i} \left( \sum_{k} v_{1k} \overline{v_{2k}} \right) \overline{u_{2i}} = \sum_{i} u_{1i} \overline{u_{2i}} \sum_{k} v_{1k} \overline{v_{2k}} = \langle u_{1}, u_{2} \rangle \langle v_{1}, v_{2} \rangle
$
\end{proof}

 \begin{lemma}

Let  B $\in \mathbb{C}^{m \times n}$.  If  $\langle B, \mathbb{C}^m \otimes v \rangle = 0$,  then $\langle \text{Row}(B), v \rangle = 0.$ Similarly

If  $\langle B, u \otimes \mathbb{C}^n \rangle = 0$,  then  $\langle \text{Col}(B), u \rangle = 0$.

\end{lemma}
\begin{proof}
    Let ${e_1,\cdots ,e_m}$ be the canonical basis for $\mathbb{C}^m$ then $\langle B,e_i\otimes v\rangle=0$ $\forall i=1,\cdots,m$. Then we have $\mathrm{Tr}(B(\overline{v}\otimes e_i))=0=\mathrm{Tr}(B(0,\cdots ,\overline{v} ,\cdots0))=\langle B^i,v\rangle$ where $B^i$ is the i-th row of B. By a similar argument, if $\langle B, u \otimes \mathbb{C}^n \rangle = 0$, then $\langle \text{Col}(B), u \rangle = 0$.

\end{proof}
\begin{theorem}[\textbf{Eckart-Young for HD}]\label{gg}
Let $A \in \mathbb{C}^{m \times n}$ be a matrix of rank $r$, and let its singular value decomposition be
$$
A = U \Sigma V^*, \qquad 
\Sigma = \mathrm{diag}(s_1, \dots, s_r, 0, \dots, 0),
$$
with singular values $s_1 \geq \cdots \geq s_r > 0$.  

For $1 \leq i \leq r$, define 
$$
\Sigma_i = \mathrm{diag}(0, \dots, 0, s_i, 0, \dots, 0),
$$
so that
$$
A = U \Sigma_1 V^* + U \Sigma_2 V^* + \cdots + U \Sigma_r V^*.
$$

Then the critical points of the Hermitian distance  function from $A$ to the variety of matrices of rank at most $k$ are precisely the matrices of the form
$$
U \bigl(\Sigma_{i_1} + \cdots + \Sigma_{i_k}\bigr) V^*, 
\qquad 1 \leq i_1 < \cdots < i_k \leq r.
$$

In particular, if the nonzero singular values of $A$ are distinct, then for each $k \leq r$ there are exactly $\binom{r}{k}$ such critical points.
\end{theorem}

Before proving this theorem, we will talk about how the secant varieties of rank one matrices are defined which gives a geometric interpretation for the rank of matrices and how their tangent spaces behave.
\begin{definition}
    Let \( X \subset \mathbb{P}V \) be an irreducible variety. The \( k \)-secant variety of \( X \) is defined by:

\[
\sigma_k(X) := \overline{
\bigcup_{\substack{p_1, \dots, p_k \in X}} 
\mathbb{P} \text{Span} \{ p_1, \dots, p_k \}
}
\]

where \( \mathbb{P} \text{Span} \{ p_1, \dots, p_k \} \) is the smallest projective linear space containing \( p_1, \dots, p_k \), and the overbar denotes the Zariski closure.
\end{definition}
If we consider $M_1$ to be the rank one matrices over a field $K$ or decomposable tensors in $K^m\otimes K^n$ then $\sigma_r(M_1)$ is just the variety of matrices with rank less than or equal to r. This is due to the fact any rank r matrix can be written as sum of rank one matrices. The tangent spaces to the secant varieties can be understood with the help of the following lemma due to Terracini.
\begin{lemma} [Terracini]
    Let \( z \in \mathbb{P} \text{Span} \{ p_1, \dots, p_k \} \) where  $ p_1, \dots, p_k $ be general points. Then:

\[
T_z \sigma_k(X) = \mathbb{P} \text{Span} \{ T_{p_1} X, \dots, T_{p_k} X \}
\]

\end{lemma}
This lemma helps us understand the tangent spaces of $M_r$ (the variety of matrices with rank at most r) with the help of the singular value decomposition. Consider a curve $\gamma(t)=v(t)\otimes u(t)$ with $\gamma(0)=v \otimes u$. The derivative at t=0 is given by $u'(0)\otimes v$ + $u \otimes v'(0)$ where $u'(0)$ and $v'(0)$ are arbitrary vectors so the tangent space is given by $\mathbb{C}^m\otimes u+v\otimes \mathbb{C}^n$ at $u\otimes v$. Now the tangent space to \( M_r \) at a point \( U(\Sigma_1 + \cdots + \Sigma_r)V^* \)=$\sum_{i=1}^r s_i \, (u_i \otimes v_i^*).$ (where $u_i,v_i$ are column vectors of U and V) can be described, by the Terracini Lemma, as

$
T_{U(\Sigma_1 + \cdots + \Sigma_r)V^*}M_r = T_{U\Sigma_1V^*}M_1 + \cdots + T_{U\Sigma_rV^*}M_1 = T_{s_1u_1 \otimes v^*_1}M_1 + \cdots + T_{s_ru_r \otimes v^*_r}M_1 = (\mathbb{C}^m \otimes v^*_1 + u_1 \otimes \mathbb{C}^n) + \cdots + (\mathbb{C}^m \otimes v^*_r + u_r \otimes \mathbb{C}^n)$. Now we can show the proof of theorem \ref{gg}.
\begin{proof}[Proof of Theorem~\ref{gg}]

 The matrix \( U(\Sigma_{i_1} + \cdots + \Sigma_{i_k}) V^* \) is a critical point of the distance function from \( A \) to the variety \( M_k \) if and only if the vector 

\[
A - U(\Sigma_{i_1} + \cdots + \Sigma_{i_k}) V^*
\]

is orthogonal to the tangent space.  \[
T_{U(\Sigma_{i_1} + \cdots + \Sigma_{i_k}) V^*} M_k = (\mathbb{C}^m \otimes{ v^*_{i_1}}+ u_{i_1} \otimes \mathbb{C}^n) + \cdots + (\mathbb{C}^m \otimes {v^*_{i_k}} + u_{i_k} \otimes \mathbb{C}^n).
\]

From the SVD of \( A \), we have 

\[
A - (U(\Sigma_{i_1} + \cdots + \Sigma_{i_k}) V^*) = U(\Sigma_{j_1} + \cdots + \Sigma_{j_l}) V^* = s_{j_1} u_{j_1} \otimes v^*_{j_1} + \cdots + s_{j_l} u_{j_l} \otimes v^*_{j_l}
\]

where \( \{j_1, \ldots, j_l\} \) is the set of indices given by the difference \( \{1, \ldots, r\} \setminus \{i_1, \ldots, i_k\} \). Let $e_1,\dots,e_m$ be the canonical basis of $\mathbb{C}^m$, then by Lemma \ref{a} We have \[
\langle s_{j_h} u_{j_h} \otimes v^*_{j_h}, e_l \otimes v^*_{i_s} \rangle = s_{j_h} \langle u_{j_h}, e_l \rangle \langle v_{j_h}, v_{i_s} \rangle = 0
\]
since \( v_{j_h} \) and \( v_{i_s} \) are distinct columns of the unitary matrix \( V \). Thus, the matrices \( U\Sigma_{j_h} V^* \) are orthogonal to the spaces \( \mathbb{C}^m \otimes V^*_{i_s} \).

In a similar way, since \( U \) is a unitary matrix, the matrices \( U\Sigma_{j_h} V^* \) are orthogonal to the spaces \( u_{i_s} \otimes \mathbb{C}^n \). Therefore, \( A - (U(\Sigma_{i_1} + \cdots + \Sigma_{i_k}) V^*) \) is orthogonal to the tangent space, and \( U(\Sigma_{i_1} + \cdots + \Sigma_{i_k}) V^* \) is a critical point. Now we will show all critical points are of this form, Let $B\in M_k$ be such that $A-B$ is orthogonal to the tangent space $T_B(M_k)$. We consider the SVD of B and A-B,let \( B = U'(\Sigma_1' + \cdots + \Sigma_k')V'^* \), and $
A - B = U''(\Sigma_1'' + \cdots + \Sigma_l'')V''^*
$
with \( \Sigma'_k \neq 0 \) and \( \Sigma''_l \neq 0 \). Since A-B is orthogonal to $T_B(M_k)$ we get that $\langle Col(A-B),u_s'\rangle=0=\langle Row(A-B),v_s'\rangle$ for $s= 1 ,\dots, k$. In particular, \(\text{Col}(A - B)\) is a vector subspace of \(\text{Span}\{u'_1, \ldots, u'_k\}^\perp\) and has dimension at most \(m - k\), while \(\text{Row}(A - B)\) is a vector subspace of \(\text{Span}\{\overline{v}'_1, \ldots, \overline{v}'_k\}^\perp\) and has dimension at most \(n - k\), so that 

$
l \leq \min\{m, n\} - k.
$
From the equality 

\[
A - B = (u''_1, \dots, u''_l, 0, \dots, 0)(\Sigma''_1 + \cdots + \Sigma''_l)V''^*
\]

we get 

\[
\text{Col}(A - B) \subset \text{Span}\{u''_1, \dots, u''_l\}
\]

and equality holds because the rank of A-B is l. In a similar way, 

\[
\text{Row}(A - B) = \text{Span}\{\overline{v}''_1, \dots, \overline{v}''_l\}.
\]

This implies that the orthonormal columns \( u''_1, \dots, u''_l, u'_1, \dots, u'_k \) can be completed with orthonormal \( m - l - k \) columns of \( \mathbb{C}^m \) to obtain a unitary \( m \times m \) matrix \( U \), while the orthonormal columns \( v''_1, \dots, v''_l, v'_1, \dots, v'_k \) can be completed with orthonormal \( n - l - k \) columns of \( \mathbb{C}^n \) to obtain a unitary \( n \times n \) matrix \( V \). We get 

\[
A - B = U
\begin{pmatrix}
\Sigma'' & 0 & 0 \\
0 & 0 & 0 \\
0 & 0 & 0
\end{pmatrix}
V^*,
\]

\[
B = U
\begin{pmatrix}
0 & 0 & 0 \\
0 & \Sigma' & 0 \\
0 & 0 & 0
\end{pmatrix}
V^*,
\]

where 

\[
\Sigma'' = \mathrm{diag}(s''_1, \dots, s''_l), \quad \Sigma' = \mathrm{diag}(s'_1, \dots, s'_k).
\]
 Thus $B$ is obtained by keeping exactly $k$ of the rank one summands $U\Sigma_iV^*$ in the SVD expansion of $A$, hence it has the required form.

\end{proof}
\begin{lemma}\label{Complex}
   Assume we have a complex subspace $W$ of a vector space V and a Hermitian inner product. If $v \in V$ and $\mathrm{Re} <v,w>=0$ $\forall w \in W$ then $<v,w>=0$   $\forall w \in W$
\end{lemma}
\begin{proof}
    Since W is closed under multiplication by i, Consider $\mathrm{Re} <v,iw>=-Im<v,w>=0$. Hence $<v,w>=0$
\end{proof}
\begin{remark}
If $T_XM$ is a complex linear subspace of $V$, then the condition
$\mathrm{Re}\langle Y-X, A\rangle = 0$ for all $A\in T_XM$ implies
$\langle Y-X, A\rangle = 0$ for all $A\in T_XM$ (Lemma~\ref{Complex}).
If $T_XM$ is only real linear, this implication need not hold (see Remark \ref{unit}).
\end{remark}

\section{Matrices invariant under unitary actions and critical points }

Let $M\subset \mathbb{C}^{n\times t}$ be an algebraic variety of matrices invariant under the left and right actions of unitary matrices, i.e.
\[
U M V^* = M \qquad \text{for all } U\in U(n),\; V\in U(t),
\]
and, without loss of generality, assume $n\le t$.

\begin{definition}
Let $\Pi^\pm_n$ be the group of signed permutations, with cardinality $2^n n!$.
A set \( S \subseteq \mathbb{R}^n \) is \emph{absolutely symmetric} if
\[
S = \pi S \quad \text{for all } \pi \in \Pi^\pm_n.
\]
For any set \( S \subseteq \mathbb{R}^n \), its \emph{absolute symmetrization} is
\[
\Pi^\pm_n S := \{ \pi x : \pi \in \Pi^\pm_n,\ x \in S \}.
\]
If $M$ is a Unitary-invariant matrix variety, define
\[
\sigma(M):=\{x\in \mathbb{R}^n:\mathrm{diag}(x)\in M\}.
\]
If $S$ is absolutely symmetric, define
\[
\sigma^{-1}(S):=\{U\,\mathrm{diag}(x)\,V^*: x\in S,\ U\in U(n),\ V\in U(t)\}.
\]
\end{definition}

\begin{definition}[Real ED degree]
Let $S \subset \mathbb{R}^n$ be a real algebraic set.
For a generic point $y \in \mathbb{R}^n$, consider the squared distance
function $d_y(x) = \|x - y\|^2$ restricted to $S$.
The number of real critical points of $d_y$
is finite and constant on each  chamber of $y$(i.e.\ on each connected component of the complement of the ED-discriminant, see \cite{draisma2016euclidean}).
We denote this chamber-wise constant in the chamber of y  by 
$\mathrm{\mathbb{R}EDdeg}(S;y)$ and the set of all such possible values by
$\mathrm{\mathbb{R}EDdeg}(S)\subset\N$
.

\end{definition}

\begin{theorem}
A set $M \subset \mathbb{C}^{n\times t}$ is unitary invariant if and only if there exists
an absolutely symmetric set $S\subset \mathbb{R}^n$ such that
\[
S=\sigma(M)
\qquad\text{and}\qquad
M=\sigma^{-1}(S).
\]
\end{theorem}

\begin{proof}
Assume $M$ is unitary invariant. Let $A\in M$ and write an SVD
\[
A=U\Sigma V^*,
\qquad
\Sigma=\mathrm{diag}(\sigma_1,\ldots,\sigma_n),
\]
(with $n\le t$ so $\Sigma\in \mathbb{C}^{n\times t}$ in the usual rectangular-diagonal sense).
Since $M$ is invariant, $\Sigma=U^*AV\in M$, hence $(\sigma_1,\ldots,\sigma_n)\in \sigma(M)=S$.
Moreover, for any signed permutation $\pi\in \Pi_n^\pm$, we can realize $\pi$ by multiplying
$\Sigma$ on the left and right by suitable unitary signed permutation matrices; therefore
$\mathrm{diag}(\pi\sigma)\in M$ and hence $\pi\sigma\in S$. Thus $S$ is absolutely symmetric.

Conversely, assume $S$ is absolutely symmetric and set $M:=\sigma^{-1}(S)$.
If $A\in M$, then $A=U\mathrm{diag}(x)V^*$ for some $x\in S$ and unitary $U,V$.
For any $U'\in U(n)$ and $V'\in U(t)$ we have
\[
U'AV'^*=(U'U)\mathrm{diag}(x)(V'V)^*\in \sigma^{-1}(S)=M,
\]
since products of unitary matrices are unitary. Hence $M$ is unitary invariant.
\end{proof}

    \begin{example}[Unitary Group]
Consider first the set 
\[
M = U(n) \subset \mathbb{C}^{n\times n}.
\]
It is clearly invariant under left and right multiplication by unitary matrices.  
Since every unitary matrix has all singular values equal to \(1\), the corresponding absolutely
symmetric set is
\[
S = \{ x \in \mathbb{R}^n : x_i = 1 \text{ for all } i \}.
\]
For any \(A \in U(n)\), an SVD has the form \(A = U \Sigma V^*\) with \(\Sigma = I_n\).
Because \(\Sigma \in \sigma^{-1}(S)\) and products of unitary matrices are Unitary, we obtain
\(\sigma^{-1}(S) = U(n)\), as expected.
\end{example}

\begin{example}[Rank Variety]
Consider the determinantal variety
\[
M_r = \{ A \in \mathbb{C}^{n\times t} : \mathrm{rank}(A) \le r \}.
\]
This set is also invariant under left and right unitary actions.
A matrix has rank at most \(r\) precisely when its singular-value vector has at most \(r\) nonzero
entries, so the corresponding absolutely symmetric set is
\[
S = \{ x \in \mathbb{R}^n : \#\{i : x_i \neq 0\} \le r \}.
\]
Hence \(\sigma^{-1}(S)\) consists exactly of the matrices whose singular values lie in this set,
and therefore \(\sigma^{-1}(S) = M_r\).
\end{example}

       \begin{lemma}\label{tsp}

           The tangent space at a point $M$ $\in U(n)$ is given by $T_{M}(U(n))=\lbrace ZM$ : $Z=-Z^* \rbrace=\{MZ: Z=-Z^*\}$
           
       \end{lemma}
   \begin{proof}
       Consider F: $\mathbb{C}^{n \times n} \to \mathbb{C}^{n\times n}$ given by $M \to MM^*$. Consider $B\in \mathbb{C}^{n\times n}$ and $t \in \mathbb{R}$. $$(M+tB)(M+tB)^*=MM^* +t(BM^* +MB^*) +t^2BB^*$$.
       Therefore $T_{M}(U(n))$ is given by the kernel of the map $B\to MB^*+BM^*$. Consider Z=$MB^*$, then $Z=-Z^*$. If we consider the map $M\to M^*M$ we get the second representation of the tangent space as well.
       
   \end{proof}

   \begin{lemma} \label{RD}
       Consider a matrix A $\in \mathbb{C}^{n \times t}$ and let D=$\mathrm{diag}(d)\in\C^{n\times t}$ with $d_i\neq 0$ and $|d_i|^2$ pairwise distinct. If $AD^*$ and $D^*A$ are Hermitian then A is a diagonal matrix. In particular if D is a real diagonal matrix then A is a real diagonal matrix
   \end{lemma}
\begin{proof}
    $AD^*$ being Hermitian gives us that $\overline{A_{ij}}d_j=A_{ji}\overline{d_i}$ and $D^*A$ being Hermitian gives us that $\overline{A_{ij}}d_i=A_{ji}\overline{d_j}$ for i,j =1 to n and $A_{ij}\overline{d}_i=0$ for i=1 to n and j$>$n. This implication in particular gives if entries of D is real so is the entries of A. Therefore we have $$\overline{A_{ij}}d_id_j=A_{ji}d_i \overline{d_i}$$ $$\overline{A_{ij}}d_id_j=A_{ji}d_j \overline{d_j}$$ Comparing these identities gives
$A_{ij}(|d_i|^2-|d_j|^2)=0$ for all $i,j$. But since $d_i \overline{d_i}$ are distinct for each i we have that $A_{ij}=0$ for i$\neq$j and since $d_i \neq 0$ we have $A_{ij}=0$ for i = 1 to n and j$>$n. Thus we have that A is diagonal.
\end{proof}
\begin{theorem}\label{sim}
    Fix an $U(n) \times U(t)$-invariant matrix variety $M \subseteq \mathbb{C}^{n \times t}$. Consider a matrix $Y \in \mathbb{C}^{n \times t}$ such that the eigenvalues of $Y Y^*$ are nonzero and distinct. Let Y=$UDV^*$ be an SVD of Y with U,V unitary ,then any HD critical point $X$ of $Y$ with respect to $M$ admits a decomposition X$=UAV^*$ where A is a real diagonal matrix.
\end{theorem}
\begin{proof}
    Consider a singular value decomposition of $Y=UDV^*$ for U $\in U(n)$, V$\in U(t)$ and D $\in \mathbb{C}^{n\times t}$. Let X be an HD critical point of Y with respect to $M$. Let A=$U^*XV$, consider the map from F : $U(n) \to M$ given by W $\to WAV^*$
    \[
[\nabla F(U)](B) = B A V^* \in T_{X}(M),
\]
for any \( B \in T_{U}(U(n)) \). By Lemma \ref{tsp}, we may write \( B = UZ \) for a skew-Hermitian matrix \( Z \), 
\[
UZ A V^* \in T_{X}(M).
\]
 we see that the tangent space of \( M \) at \( X \) contains
\[
\{ UZ A V^* : Z^*= -Z \}.
\]
Then, by the definition of an HD critical point, we have
\[
\mathrm{Re}\mathrm{Tr}\left((Y - X)(UZ A V^*)^*\right) = 0
\]
for any skew-Hermitian matrix \( Z \).
\[
0 = \mathrm{Re}\mathrm{Tr}\left(U(D - A)V^*V A^* Z^* U^*\right) = \mathrm{Re}\mathrm{Tr}\left((D - A)A^* Z^*\right).
\]
 by Lemma \ref{skew} this means $(D-A)A^*$ is Hermitian also note $(D-A)A^*$ is 0 when the imaginary part of the trace vanishes as well. Since $AA^*$ is Hermitian, we
have that $DA^*$ is Hermitian; therefore the transpose $AD^*$ is Hermitian. Similarly we will show $D^*A$ is Hermitian
By considering \( F : U(t) \to M \) given by \( W \mapsto U A W^* \), we get, as above, that  
\[
\{ U A Z^* V^* : Z^* = -Z \} \subseteq T_{X}(M).
\]  
It follows that  
\[
0 = \mathrm{Re}\mathrm{Tr}\left( (U(D - A)V^*)^* U A Z^* V^* \right) = \mathrm{Re}\mathrm{Tr}\left( (D - A)^* A Z^* \right)
\]  
for any skew-Hermitian matrix \( Z \), \( (D - A)^* A \) is Hermitian
.  
Again, since \( A^* A \) is Hermitian, we get that \( D^* A \) is Hermitian.  
Since \( A D^* \) and \( D^* A \) are both Hermitian, we conclude that \( A \) is real diagonal by Lemma \ref{RD}.
\end{proof}
\begin{lemma}\label{des}
Let $S \subseteq \mathbb{R}^n$ be a $\Pi_n^{\pm}$-invariant real algebraic set and set
\[
M:=\sigma^{-1}(S)
=\{\,U\mathrm{diag}(x)V^* : x\in S,\ U\in U(n),\ V\in U(t)\,\}.
\]
If $Q\subseteq S$ is Euclidean dense in $S$, then
\[
N_Q:=\{\,U\mathrm{diag}(x)V^* : x\in Q,\ U\in U(n),\ V\in U(t)\,\}
\]
is Euclidean dense in $M$.
\end{lemma}

\begin{proof}
Fix $A=U\mathrm{diag}(x)V^*\in M$. Since $Q$ is dense in $S$, pick $x_k\in Q$ with $x_k\to x$.
For fixed $U,V$, the map
\[
\phi_{U,V}: \mathbb{R}^n \to \mathbb{C}^{n\times t}, \qquad z \mapsto U\mathrm{diag}(z)V^*
\]
is continuous, hence $\phi_{U,V}(x_k)\to \phi_{U,V}(x)=A$.
\end{proof}

\begin{lemma}\label{as}\cite[Theorem~4.8]{costeSAG}
Let $M \subset \R^n$ and $N \subset \R^p$ be semialgebraic smooth submanifolds, and let
$f : M \to N$ be a semialgebraic $C^\infty$ mapping. Then the set of critical values of $f$
is a semialgebraic subset of $N$, of dimension $< \dim N$.
\end{lemma}

\begin{lemma}\label{tsM}
Consider a $\Pi_n^{\pm}$-invariant real algebraic set $S \subseteq \mathbb{R}^n$ and the induced variety
$M := \sigma^{-1}(S)\subseteq \C^{n\times t}$ (viewed inside $V^{\R}\cong \R^{2nt}$).
Then there exists a semialgebraic open dense subset $M^\circ\subseteq M^{\mathrm{reg}}$ such that:
for every $X\in M^\circ$ and every decomposition $X=U\mathrm{diag}(x)V^*$ with $x\in S^{\mathrm{reg}}$, one has
\[
T_{X}(M) = \left\{ UZ_1 \mathrm{diag}(x) V^{*} + U \mathrm{diag}(x) Z_2^{*} V^{*} + U \mathrm{diag}(a) V^{*}
\,:\, a \in T_{x}(S),\; Z_1, Z_2 \text{ skew-Hermitian} \right\}.
\]
\end{lemma}

\begin{proof}
 $V^{\R}\cong \R^{2nt}$, so that $U(n)$ and $U(t)$ are smooth real algebraic
submanifolds (defined by polynomial equations in real and imaginary parts). Consider the semialgebraic
$C^\infty$ map
\[
\Phi: U(n)\times S^{\mathrm{reg}}\times U(t)\longrightarrow V^{\R},\qquad
(U,x,V)\longmapsto U\,\mathrm{diag}(x)\,V^* .
\]
Its image is contained in $M$, and it is Euclidean dense in $M$ ( Lemma \ref{des}).

Apply Lemma~\ref{as} to $\Phi$ with target manifold $N:=M^{\mathrm{reg}}$:
the set of critical values $\mathrm{CritVal}(\Phi)\subseteq M^{\mathrm{reg}}$ is semialgebraic of
dimension $<\dim M^{\mathrm{reg}}$. Hence
\[
M^\circ := M^{\mathrm{reg}}\setminus \mathrm{CritVal}(\Phi)
\]
is semialgebraic open dense in $M^{\mathrm{reg}}$. For any $X\in M^\circ$ and any $(U,x,V)$ with
$\Phi(U,x,V)=X$, the differential $d\Phi_{(U,x,V)}$ is surjective onto $T_X(M)$.

Now compute $d\Phi$. We have
\[
T_UU(n)=\{UZ_1: Z_1^*=-Z_1\},\qquad T_VU(t)=\{VZ_2: Z_2^*=-Z_2\},
\qquad T_xS^{\mathrm{reg}}=T_xS.
\]
Therefore for $(UZ_1,a,VZ_2)\in T_UU(n)\times T_xS\times T_VU(t)$,
\begin{align*}
d\Phi_{(U,x,V)}(UZ_1,a,VZ_2)
&= UZ_1\,\mathrm{diag}(x)\,V^* \;+\; U\,\mathrm{diag}(a)\,V^*
\;+\; U\,\mathrm{diag}(x)\,(VZ_2)^* \\
&= UZ_1\,\mathrm{diag}(x)\,V^* \;+\; U\,\mathrm{diag}(a)\,V^*
\;+\; U\,\mathrm{diag}(x)\,Z_2^*\,V^* .
\end{align*}
Since $d\Phi_{(U,x,V)}$ is surjective, its image equals $T_X(M)$.
\end{proof}
\begin{lemma}[Generic data avoid the bad tangent locus]\label{genericg}
Let $M \subset V^{\R}$ be a real algebraic set and let
$M^\circ \subset M^{\mathrm{reg}}$ be the open dense subset from
Lemma~\ref{tsM}. Then there exists a semialgebraic open dense set
$\Omega \subset V^{\R}$ such that for every $Y \in \Omega$,
all HD critical points of $Y$ on $M$ lie in $M^\circ$.
\end{lemma}

\begin{proof}
For $X \in M^{\mathrm{reg}}$, the HD critical condition is
$Y-X \in (T_XM)^{\perp_q}$.
For fixed $X$, this is an affine space of dimension
$\dim V^{\R}-\dim M$.

Let $B := M^{\mathrm{reg}}\setminus M^\circ$. By Lemma~\ref{tsM},
$\dim B < \dim M$. Consider
\[
\mathcal I=\{(X,Y)\in B\times V^{\R} : Y-X\in (T_XM)^{\perp_q}\}.
\]
Each fiber over $X$ has dimension $\dim V^{\R}-\dim M$, so
\[
\dim \mathcal I \le \dim B + (\dim V^{\R}-\dim M) < \dim V^{\R}.
\]
Hence its projection to $V^{\R}$ has dimension $<\dim V^{\R}$.
Outside this set, no critical point lies in $B$.
\end{proof}

\begin{lemma} \label{skew}
Let $<A,B> = \mathrm{Tr}(A B^{*})$ be the inner product on 
$M_n(\mathbb{C})$. Then:
\begin{enumerate}
\item If $\mathrm{Re}\,<A,Z> = 0$ for all skew Hermitian $Z$, then $A$ is Hermitian.
\item If $<A,Z> = 0$ for all skew Hermitian $Z$, then $A = 0$.
\end{enumerate}
\end{lemma}
\begin{proof}
Every matrix $A$ decomposes uniquely as
\[
A = H + K, \qquad 
H := \tfrac12(A + A^{*})\ \text{(Hermitian)}, \qquad
K := \tfrac12(A - A^{*})\ \text{(skew Hermitian)}.
\]

If $H$ is Hermitian and $K$ is skew Hermitian, then
\[
<H,K> = \mathrm{Tr}(H K^{*}) = -\,\mathrm{Tr}(H K).
\]
On the other hand,
\[
\overline{<H,K>}
= \overline{\mathrm{Tr}(H K^{*})}
= \mathrm{Tr}((H K^{*})^{*})
= \mathrm{Tr}(K H)
= \mathrm{Tr}((H K)^{*})
= \overline{\mathrm{Tr}(H K)}.
\]
Thus $\mathrm{Tr}(H K)$ is purely imaginary, and therefore
\[
\mathrm{Re}\,<H,K> = 0. \tag{$*$}
\]

(1) Real-part condition.
Assume $\mathrm{Re}\,<A,Z> = 0$ for all skew Hermitian $Z$.
Taking $Z = K$, we obtain
\[
0 = \mathrm{Re}\,<A,K>
= \mathrm{Re}\,<H+K,K>
= \mathrm{Re}\,<H,K> + \mathrm{Re}\,<K,K>.
\]
By $(*)$, $\mathrm{Re}\,<H,K>=0$, so
\[
0 = \mathrm{Re}\,<K,K> = \mathrm{Tr}(K K^{*}) \ge 0.
\]
Hence $K = 0$ and $A = H$ is Hermitian.

(2) Full inner-product condition.
Assume $<A,Z> = 0$ for all skew Hermitian $Z$.
If $H$ is Hermitian, then $Z = iH$ is skew Hermitian, so
\[
0 = <A,iH> = i <A,H> \quad\Rightarrow\quad <A,H> = 0
\]
for all Hermitian $H$.  
Any matrix $B$ can be written as $B = H_1 + iH_2$ with $H_1,H_2$ Hermitian, hence
\[
<A,B> = <A,H_1> + i <A,H_2> = 0.
\]
Thus $A$ is orthogonal to all matrices, so
\[
0 = <A,A> = \mathrm{tr}(A A^{*}),
\]
which implies $A = 0$.
\end{proof}

\begin{theorem}[Real ED degree of $S$ and HD degree of $M$]\label{main}
Consider a $\Pi^{\pm}_n$-invariant variety $S \subseteq \mathbb{R}^n$ and the
induced real variety $M := \sigma^{-1}(S) \subseteq \mathbb{C}^{n \times t}$.
Then for all $Y\in\C^{n\times t}$ in a semialgebraic open dense subset ,
$Y$ admits an SVD
$
Y = U\mathrm{diag}(y)V^*
$ with $U \in U(n)$, $V \in U(t)$, and $y \in \mathbb{R}^n$.
Moreover, the set of Hermitian distance critical points of $Y$ with respect to $M$ is
\[
\{\, U \mathrm{diag}(x) V^{*} : x \text{ is an ED critical point of } y 
\text{ with respect to } S \,\}.
\]
In particular, the Hermitian distance degree of $M$ equals the real Euclidean distance
degree of $S$:
\[
\mathrm{HDdeg}(M) = \mathrm{\mathbb{R}EDdeg}(S).
\]
\end{theorem}

\begin{proof}
Let $Y\in \mathbb{C}^{n \times t }$ with distinct singular values and let $Y=U\mathrm{diag}(y)V^*$ be an SVD of Y.   Let $X$ be an HD critical point of $Y$ with respect to $M$. By Lemma \ref{sim} we may assume that $X$ can be written as
\[
X = U \mathrm{diag}(x) V^*
\]
for some $x \in S$.

We can further assume that $X$ is generic and by Lemma \ref{genericg} the tangent space $T_{X}(M)$ at $X = U \mathrm{diag}(x) V^*$ is given in Lemma \ref{tsM}.

We will now show that $x$ is an ED critical point of $y$ with respect to $S$. To see this, observe the inclusion
\[
\left\{ U \mathrm{diag}(a) V^* : a \in T_{x}(S) \right\} \subseteq T_X(M),
\]
and since $y$, $x$ and $a$ are in $\mathbb{R}^n$ the trace is real and hence we have,
\[
0 = \mathrm{Tr}\left( U \mathrm{diag}(y - x) V^* \left( U \mathrm{diag}(a) V^* \right)^* \right)
\quad \text{for any } a \in T_{x}(S).
\]

therefore $0 = \mathrm{Tr} \left( U \, \mathrm{diag}(y - x) \, V^* \cdot V \, \mathrm{diag}(\bar{a}) \, U^* \right)$

\[
(y - x)^\top \overline{a} = 0 \quad \text{for any } a \in T_{x}(S),
\]
and hence $x$ is an ED critical point of $y$ with respect to $S$.
Conversely, suppose \( x \in S^{\mathrm{reg}} \) is an ED critical point of \( y \) with respect to \( S \).

Define now the matrix \( X := U \, \mathrm{diag}(x) \, V^* \). 
If moreover $X\in M^\circ$ (the open dense set from Lemma~\ref{tsM}),
then the tangent space at $X$ has the stated form and is generated by:

\begin{enumerate}
  \item \( U Z \mathrm{diag}(x) V^* \), with \( Z \) skew-Hermitian,
  \item \( U \mathrm{diag}(a) V^* \), where \( a \in T_{x}(S) \),
  \item \( U \mathrm{diag}(x) Z^* V^* \), with \( Z \) skew-Hermitian.
\end{enumerate}

We will show for each generator G as in (i) (ii) (iii), we have the $\mathrm{Re}  (\mathrm{Tr}((X-Y)G^*))$ vanishes

For (i) and (iii), observe since x,y and a are all in $\mathbb{R}^n$ and $\mathrm{Re}(\mathrm{Tr}(AB^*))=0$ for A Hermitian and B skew Hermitian:
\begin{equation}
\begin{aligned}
\mathrm{Re} \mathrm{Tr}\!\left( (X - Y)(U Z_1 \mathrm{diag}(x) V^*)^* \right)
&= \mathrm{Re} \mathrm{Tr}\!\left( \mathrm{diag}(x - y)\mathrm{diag}({x})^* Z_1^* \right) = 0 \\
&= \mathrm{Re} \mathrm{Tr}\!\left(\mathrm{diag}(x - y)\mathrm{diag}(x)Z_2 \right) \\
&= \mathrm{Re} \mathrm{Tr}\!\left( (X - Y)(U \mathrm{diag}(x)Z_2^* V^*)^* \right)
\end{aligned}
\label{comp}
\end{equation}
In case (ii) we have from the fact x is an ED critical point for y on S $$
\mathrm{Tr}\left((X - Y)(U \mathrm{diag}(a) V^*)^*\right)=\mathrm{Tr}(U\mathrm{diag}(x-y)\mathrm{diag}(a)^*U^*)
= \mathrm{Tr}\left(\mathrm{diag}(x - y) \mathrm{diag}(a)^*\right) = 0
$$
therefore we have the result $\mathrm{HDdeg}(M) = \mathrm{\mathbb{R}EDdeg}(S).$
 
\end{proof}

This theorem gives us the Hermitian analogue for the main result in \cite{drusvyatskiy2017euclidean}.\\
\begin{example}\label{paras}[Singular-Value Relation \(\{\sigma_1,\sigma_2\} = \{a,a^2\}\)]
For \(2\times 2\) matrices whose singular values satisfy
\[
\{\sigma_1, \sigma_2\} = \{a, a^2\} \qquad (a \ge 0),
\]
the corresponding absolutely symmetric set is
\[
S = \Pi^{\pm} \cdot \{ (a, a^2) : a \ge 0 \},
\]
i.e.\ the set of all signed permutations of the pair \((a,a^2)\).  
Equivalently, \(S\) may be described as the real algebraic set
\[
S = \{ (x_1, x_2) \in \mathbb{R}^2 :
(x_2 - x_1^{2})(x_1 - x_2^{2}) = 0 \},
\]
which is invariant under permutations and sign changes of the coordinates. 

Let $E_1$ (resp.\ $E_2$) be the evolute of the parabola $x_2=x_1^2$ (resp.\ $x_1=x_2^2$)(see figure \ref{fig:parabolas-evolutes}).
For a generic point $y$, the number of real ED critical points contributed by a given parabola
is $1$ if $y$ lies on the same side of its evolute as the parabola, and it is $3$ if $y$ lies on the other side.
Equivalently,
if $y$ is outside $E_1$, then $x_2=x_1^2$ contributes $1$ real critical point if $y$ is inside $E_1$, it contributes $3$. If $y$ is outside $E_2$, then $x_1=x_2^2$ contributes $1$ real critical point if $y$ is inside $E_2$, it contributes $3$.

Hence the total number of real ED critical points of $y$ on $S$ equals 2, 4, or 6, depending on
whether $y$ lies inside none, exactly one, or both of the evolutes. By Theorem \ref{main}, for a generic matrix with SVD $Y=U\mathrm{diag}(y)V^*$ the Hermitian distance critical points of $Y$ on
$M=\sigma^{-1}(S)$ are exactly the matrices $U\mathrm{diag}(x)V^*$ where $x$ ranges over these real ED critical points of $y$ on $S$.
Hence the HD critical count for $Y$ equals the ED critical count for $y$, and it jumps precisely when $y$ crosses the evolutes.

\end{example}
\begin{figure}[ht]
\centering
\begin{tikzpicture}
\begin{axis}[
  axis lines=middle,
  xmin=-2, xmax=6,
  ymin=-2, ymax=6,
  samples=120,
  xlabel={$x_1$}, ylabel={$x_2$},
  width=0.85\textwidth,
  height=0.7\textwidth,
  legend style={at={(0.02,0.98)},anchor=north west},
]

\addplot[thick, blue, domain=-2.2:2.2] {x^2};

\addplot[thick, green!60!black, domain=0:6] {sqrt(x)};
\addplot[thick, green!60!black, domain=0:6] {-sqrt(x)};

\addplot[thick, red, domain=0.5:6]
  ({ sqrt(16*(x-0.5)^3/27) }, x);
\addplot[thick, red, domain=0.5:6]
  ({-sqrt(16*(x-0.5)^3/27) }, x);

\addplot[thick, orange, domain=0.5:6]
  (x, { sqrt(16*(x-0.5)^3/27) });
\addplot[thick, orange, domain=0.5:6]
  (x, {-sqrt(16*(x-0.5)^3/27) });

\end{axis}
\end{tikzpicture}
\caption{The absolutely symmetric set 
$S=\{(x_1,x_2):(x_2-x_1^2)(x_1-x_2^2)=0\}$ (union of two parabolas) together with the evolutes.}
\label{fig:parabolas-evolutes}
\end{figure}
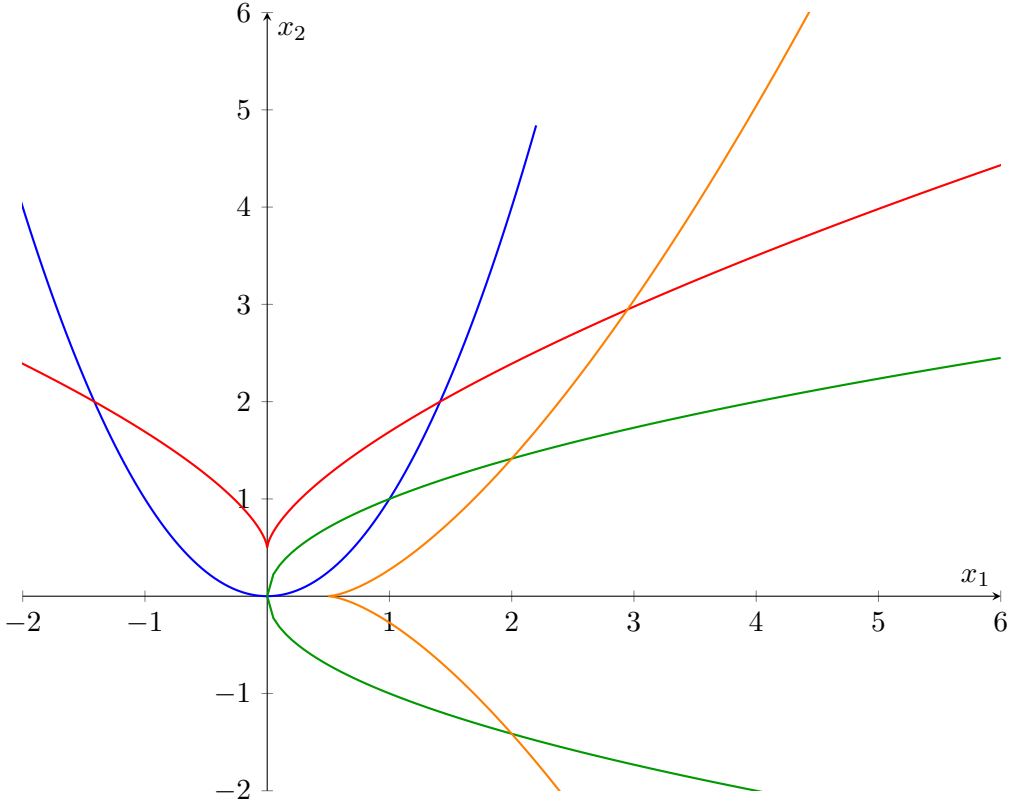
\begin{example}[Determinant-magnitude constraint in $2\times2$]\label{ex:det-mag-2x2}
Let
\[
M_{\det}:=\{A\in\C^{2\times2}:\ |\det A|=1\},
\]
viewed as a real algebraic set in $V^{\R}\cong\R^{8}$, equipped with
$q=\mathrm{ReTr}(AB^*)$.  This set is invariant under the left-right action of
$U(2)\times U(2)$ since $|\det(UAV^*)|=|\det A|$.

Its diagonal slice is
\[
S=\sigma(M_{\det})
=\{x\in\R^2:\ |\det(\mathrm{diag}(x))|=1\}
=\{x\in\R^2:\ |x_1x_2|=1\}
=H_2^+\cup H_2^-,
\]
where $H_2^{\pm}=\{x_1x_2=\pm1\}$.

Fix $y=(y_1,y_2)\in\R^2$ and parametrize the two branches by
\[
H_2^+:\ (t,1/t),\qquad H_2^-:\ (t,-1/t)\qquad (t\neq0).
\]
A point on $H_2^{\pm}$ is ED-critical for $y$ iff $t$ is a real root of
\[
q_y^{\pm}(t):=t^4-y_1t^3\pm y_2t-1.
\]
Let $D^{\pm}(y)$ be the discriminant of $q_y^{\pm}(t)$ as a polynomial in $t$:
\[
D^{+}(y)= -256 + 192y_1y_2 + 6y_1^2y_2^2 + 4y_1^3y_2^3 - 27y_1^4 - 27y_2^4,
\]
\[
D^{-}(y)= -256 - 192y_1y_2 + 6y_1^2y_2^2 - 4y_1^3y_2^3 - 27y_1^4 - 27y_2^4.
\]
For $y$ not on the semialgebraic curve $\{D^+(y)D^-(y)=0\}$, the number of real
ED-critical points of $y$ on $S=H_2^+\cup H_2^-$ is
\[
\mathrm{\mathbb{R}EDdeg}(S:y)=
\begin{cases}
6, & \text{if } D^+(y)>0 \text{ or } D^-(y)>0,\\
4, & \text{if } D^+(y)<0 \text{ and } D^-(y)<0,
\end{cases}
\]
see \cite[Example 4.4]{drusvyatskiyLeeThomas2015}.

Now let $Y\in\C^{2\times2}$ be generic and write an SVD $Y=U\mathrm{diag}(y)V^*$.
By Theorem~\ref{main}, the HD-critical points of $Y$ on $M_{\det}$ are exactly
\[
\{\,U\mathrm{diag}(x)V^*:\ x\in S \text{ is ED-critical for } y\,\},
\]
so the chamber-wise Hermitian critical count for $M_{\det}$ is $\{4,6\}$, and
it jumps precisely when $y$ crosses $\{D^+(y)D^-(y)=0\}$.
\end{example}

    \begin{example}[Unit Schatten $d$-sphere]\label{ex:schatten-sphere}
Fix an even integer $d\ge 2$. For $X\in\C^{n\times t}$, define the Schatten $d$-norm
\[
\|X\|_{d}:=\Big(\sum_{i=1}^n \sigma_i(X)^d\Big)^{1/d},
\]
where $\sigma_1(X),\dots,\sigma_n(X)$ are the singular values of $X$.
Consider the (Unitary-invariant) unit sphere
\[
\mathcal F_{n,t,d}:=\{X\in\C^{n\times t}:\ \|X\|_{d}=1\}.
\]
Since $d$ is even, $\mathcal F_{n,t,d}$ is a real algebraic set in $V^{\R}\cong\R^{2nt}$.

Its singular-value image is the affine Fermat hypersurface
\[
F_{n,d}:=\Big\{x\in\R^{n}:\ \sum_{i=1}^n x_i^{d}=1\Big\},
\]
which is absolutely symmetric. Hence $\mathcal F_{n,t,d}=\sigma^{-1}(F_{n,d})$.

Let $Y\in\C^{n\times t}$ be generic and write an SVD $Y=U\mathrm{diag}(y)V^*$ with $y\in\R^n$.
By Theorem~\ref{main}, the Hermitian distance critical points of $Y$ on $\mathcal F_{n,t,d}$
are precisely the matrices $U\mathrm{diag}(x)V^*$ where $x$ ranges over the real ED critical points of $y$
on $F_{n,d}$. In particular,
\[
\mathrm{HDdeg}(\mathcal F_{n,t,d})=\mathrm{\mathbb{R}EDdeg}(F_{n,d}).
\]

In the case $n=2$, write $F_{2,d}=\{(x_1,x_2)\in\R^2:\ x_1^d+x_2^d=1\}$.
A point $x\in F_{2,d}$ is a real ED critical point of $y=(y_1,y_2)$ if and only if
it lies on the curve $\gamma_{d,y}$.
\[
\gamma_{d,y}:=\Big\{x\in\R^2:\ x_1^{d-1}(x_2-y_2)=x_2^{d-1}(x_1-y_1)\Big\}.
\]
see \cite[Example 4.3]{drusvyatskiyLeeThomas2015}.

For $d=2$ one has $\#(F_{2,2}\cap\gamma_{2,y})=2$ for generic $y$ (the circle case).
For every even $d\ge 4$, there exists a nonempty open set of $y\in\R^2$ for which
$\#(F_{2,d}\cap\gamma_{d,y})=8$, hence $\mathrm{HDdeg}(\mathcal F_{2,t,d})=8$ in that generic chamber.
\end{example}

\begin{remark} \label{unit}

Consider the unitary group U(n) then the absolute symmetric set S is given by $(\pm 1,....,\pm 1)$, since each point is critical in the variety the ED degree of S is $2^n$. Therefore the HD degree of U(n) is the same. If $A$ a general matrix has an SVD given by $U\Sigma V$ then the corresponding critical points are obtained at $U\mathrm{diag}(x)V$ such that $x \in S$.\\ Let
$A = \begin{bmatrix}
i & 0 & 0 \\
0 & 2 & 0 \\
0 & 0 & 3
\end{bmatrix}$.
One of the SVD of A is given by
$A = U \Sigma V^*$,
where
$\Sigma = \mathrm{diag}(1,2,3)$, 
$U = \mathrm{diag}(i,1,1)$, 
and $V = I$. Then the critical points should admit a simultaneous decomposition and therefore should be $U\mathrm{diag}(\pm1,\pm1,\pm1)$. Note one of the critical points is U itself and from \ref{tsM}, $U\mathrm{diag}(i,-2i,3i)$ is an element of $T_UM$. Now if you compute $\mathrm{Tr}((A-U)(U\mathrm{diag}(i,-2i,3i))^*)$=-4i, so the imaginary part remains. This is due to the fact that the tangent space is real linear and the unitary group is a not a complex submanifold.
\end{remark}

\begin{remark}
    Notice if M is a complex sub manifold of $\mathbb{C}^{n \times t}$  which is invariant under unitary actions and $X=U\mathrm{diag}(x)V$ is a critical point for $Y=U\mathrm{diag}(y)V$ a data point, then in particular from (\ref{comp}), Lemma \ref{skew} and Lemma \ref{Complex}  we have $\mathrm{diag}(x-y)\mathrm{diag}(x)=0$. 
    
\end{remark}

Let $V=\C^{m\times n}$ with $m\le n$ and Hermitian inner product
$\langle A,B\rangle=\mathrm{Tr}(AB^*)$.
For $1\le r\le m$ let
\[
X_r:=\{A\in \C^{m\times n}:rank(A)\le r\}.
\]
Fix $U\in \C^{m\times n}$ and write a singular value decomposition
\[
U=T_1\Sigma T_2,\qquad \Sigma=\mathrm{diag}(\sigma_1,\dots,\sigma_m),\ \sigma_1\ge\cdots\ge\sigma_m\ge0,
\]
with $T_1\in U(m)$ and $T_2\in U(n)$.
For $1\le j\le m$ let $\Sigma_j$ denote the diagonal matrix with $\sigma_j$ in the $j$-th
diagonal position and zeros elsewhere.

Assume $U$ is generic (e.g.\ $\sigma_1>\cdots>\sigma_m$). Then the Hermitian critical points of the
squared distance function $d_U(A)=\|U-A\|^2$ on $X_r$ are precisely the matrices
\[
T_1(\Sigma_{i_1}+\cdots+\Sigma_{i_r})T_2,
\qquad 1\le i_1<\cdots<i_r\le m.
\]
Consequently, the Hermitian distance polynomial $HDpoly_{X,U}$ as in \cite{Furchi2025_HermitianDistanceDegree}, i.e.\ the Euclidean
distance polynomial of the real variety with respect to the real inner product
$\mathrm{Re}\langle\cdot,\cdot\rangle$ induced by the standard Hermitian form.
\[
HDpoly_{X_r,U}(t^2)
=
\prod_{1\le i_1<\cdots<i_r\le m}
\left(t^2-\sum_{j\notin\{i_1,\dots,i_r\}}\sigma_j(U)^2\right).
\]

In particular, for the corank one determinantal variety $X_{m-1}$ we obtain
\[
HDpoly_{X_{m-1},U}(t^2)
=
\prod_{k=1}^m (t^2-\sigma_k(U)^2)
=
\det\!\bigl(t^2I_m-UU^*\bigr).
\]
This is the Hermitian analogue of \cite[Theorem~8.1]{OS}

\newpage

\section{Slicing Theorem for Hermitian Distance Degree}
Section 3 gives an explicit description of Hermitian critical points and their critical values via the singular-value reduction. In this section we explain how the same reduction also follows abstractly from a slicing principle of Bik–Draisma. We state the result for the Hermitian distance degree, paralleling the
Euclidean distance degree result in \cite{bik2018eddegrees}.
\begin{theorem}\cite{bik2018eddegrees}[Bik-Draisma]\label{BD} Let V be a finite dimensional vector space over $\mathbb{C}$ equipped with a Euclidean product $<,>$. Let G be a complex algebraic group with an orthogonal representation O(V).
    Suppose that $V$ has a linear subspace $V_0$ such that, for sufficiently 
general $v_0 \in V_0$, the space $V$ is the orthogonal direct sum of $V_0$ 
and the tangent space $T_{v_0}(G \cdot v_0)$ to its $G$-orbit. 
Let $X$ be a $G$-stable closed subvariety of $V$. 
Set $X_0 := X \cap V_0$ and suppose that $G \cdot X_0$ is dense in $X$. 
Then the ED degree of $X$ in $V$ equals the ED degree of $X_0$ in $V_0$.

\end{theorem}

Let $V$ be a finite-dimensional complex vector space with Hermitian inner product
$\langle\cdot,\cdot\rangle$, and let
\[
q:=\mathrm{Re}\langle\cdot,\cdot\rangle
\]
be the associated real Euclidean inner product on the realification $V^{\R}$.
For a real algebraic subset $X\subseteq V^{\R}$ (in particular, for a complex variety viewed as a real set),
the critical points of the squared Hermitian distance function are exactly the \emph{real} Euclidean
distance critical points in the inner-product space $(V^{\R},q)$. Hence the relevant enumerative invariant
is a chamberwise \emph{real} critical count.

If $G\to U(V)$ is a unitary representation, then $G$ acts on $V^{\R}$ by $q$-orthogonal transformations.
Therefore the slicing method of Bik--Draisma \cite{bik2018eddegrees} applies in the real inner-product space
$(V^{\R},q)$; we include the argument for completeness, emphasizing that tangent spaces are generally only
\emph{real} linear in the Hermitian distance problem.
Throughout, $\mathrm{HDdeg}(X)$ denotes the number of real Euclidean distance critical points of a sufficiently
general data point in $V^{\R}$ with respect to $q$.

\begin{theorem}[Hermitian slicing Theorem]\label{dbrevised}
Let $V$ be a finite-dimensional complex vector space with Hermitian inner product $\langle\cdot,\cdot\rangle$,
and let $q=\mathrm{Re}\langle\cdot,\cdot\rangle$ be the associated real Euclidean inner product on $V^{\mathbb R}$.
Let $G\to U(V)$ be a unitary representation (hence an orthogonal representation on $(V^{\mathbb R},q)$).

Suppose that $V^{\mathbb{R}}$ has a real linear subspace $V_{0}^{\mathbb{R}}$ such that, for sufficiently
general $v_0 \in V_{0}^{\mathbb{R}}$, the space $V^{\mathbb{R}}$ is the $q$-orthogonal direct sum of
$V_{0}^{\mathbb{R}}$ and the tangent space $T_{v_0}(G \cdot v_0)$.
Let $X$ be a $G$-stable real subvariety of $V^{\mathbb{R}}$.
Set $X_0 := X \cap V_{0}^{\mathbb{R}}$ and suppose that $G \cdot X_0$ is dense in $X$.
Then $\mathrm{HDdeg}_{V^\mathbb{R}}(X)=\mathrm{\mathbb{R}EDdeg}_{V_0^{\R}}(X_0)$.
\end{theorem}

\begin{lemma}\cite{bik2018eddegrees}(Lemma 16)
Critical points are equivariant: if $x$ is a critical point for $u$, then $g\cdot x$ is a critical point for $g\cdot u$.\label{hdcrit}
\end{lemma}

\begin{lemma}\cite{bik2018eddegrees}(Lemma 15)
    The set $GV_0$ is dense in $V$\label{dense}
\end{lemma}
\begin{lemma}\cite{bik2018eddegrees}(Lemma 17)
   A sufficiently general $x_0 \in X_0$  lies in $X^{reg}$ and $X_0^{reg}$
\end{lemma}

\begin{lemma}\label{TSG}
A sufficiently general $x_{0} \in X_{0}$ satisfies
\[
T_{x_{0}}X \;=\; T_{x_{0}}X_{0} \;+\; T_{x_{0}}(G\cdot x_{0}).
\]
\end{lemma}

\begin{proof}
Consider the action map (on real loci)
\[
m : G \times X_{0} \longrightarrow X, \qquad (g,x)\longmapsto g\cdot x .
\]
By assumption $G\cdot X_0$ is dense in $X$, so the image of $m$ is Euclidean dense in $X$.
Restrict $m$ to the smooth loci:
\[
m^{\mathrm{reg}} : G \times X_{0}^{\mathrm{reg}} \longrightarrow X^{\mathrm{reg}},
\]
which is a semialgebraic $C^\infty$ map between semialgebraic smooth manifolds .

Apply Lemma~\ref{as} to $m^{\mathrm{reg}}$. The set of critical values in $X^{\mathrm{reg}}$ is semialgebraic
of dimension $<\dim X^{\mathrm{reg}}=\dim X$. Hence the set of regular values is dense in $X^{\mathrm{reg}}$.
Pick a point $x\in X^{\mathrm{reg}}$ that is a regular value of $m^{\mathrm{reg}}$ and lies in the dense set
$G\cdot X_{0}^{\mathrm{reg}}$. Choose $(g,x_0)\in G\times X_{0}^{\mathrm{reg}}$ with $g\cdot x_0=x$.
Since $x$ is a regular value, the differential $dm_{(g,x_0)}$ is surjective onto $T_xX$.

For $\xi\in T_gG$ and $v\in T_{x_0}X_0$,
\[
dm_{(g,x_0)}(\xi,v)=\xi\cdot x \;+\; g\cdot v,
\]
so
\[
\mathrm{Im}(dm_{(g,x_0)}) = T_x(G\cdot x) \;+\; g\cdot T_{x_0}X_0.
\]
Surjectivity gives
\[
T_xX = T_x(G\cdot x) \;+\; g\cdot T_{x_0}X_0.
\]
Apply $g^{-1}$ and use $g^{-1}\cdot T_xX=T_{x_0}X$ and
$g^{-1}\cdot T_x(G\cdot x)=T_{x_0}(G\cdot x_0)$ to obtain
\[
T_{x_0}X = T_{x_0}(G\cdot x_0) \;+\; T_{x_0}X_0.
\]
Finally, the set of such $x_0$ is dense in $X_0$,
so the equality holds for sufficiently general $x_0\in X_0$.
\end{proof}

\begin{lemma}
    Let $v_0 \in V_0^\mathbb{R}$ then any $\mathrm{\mathbb{R}ED}$ critical point on $X_0$ for $v_0$ is an HD critical point on X for $v_0$
\end{lemma}
\begin{proof}
    Let $x_0$ be a real ED critical point for $v_0$ in $X_0$ then $v_0 -x_0 \perp T_{x_0}X_0$ and since $v_0-x_0\in V_0$ which is orthogonal to $T_{x_0}Gx_0$ therefore by Lemma \ref{TSG} we have $v_0-x_0\perp T_{x_0}X$
\end{proof}
\begin{lemma}
    Let $v_{0} \in V_{0}$ be sufficiently general. Then any  HD critical point on $X$ for $v_{0}$ is
an $\mathrm{\mathbb{R}ED}$ critical point on $X_{0}$ for $v_{0}$.
\end{lemma}
\begin{proof}
Let x be an HD critical point for $v_0$ on X. Then $\mathrm{Re} <x-v_0,T_{x}X>=0$ and also we have that there exists $x_0$ such that  $T_x X = g \cdot T_{x_0} X = g \cdot T_{x_0} X_0 + g \cdot T_{x_0} (G x_0) = T_x X_0^g + T_x (G x_0)$. Since $T_x(G \cdot x_0) = g \cdot (\mathfrak{g} \cdot x_0) = \{ g \cdot (\xi \cdot x_0) \mid \xi \in \mathfrak{g} \} = \mathfrak{g} \cdot x$ we have $x-v_0 \perp \mathfrak{g} \cdot x $.    Then we have $v_0\perp \mathfrak{g} \cdot x $ and since G is a unitary representation we have $\mathrm{Re} <\mathfrak{g} \cdot x,v_0>=-\mathrm{Re} <x,\mathfrak{g}.v_0>=0. $ Since $v_0$ is sufficiently general in $V_0^\mathbb{R}$, the vector space $V^\mathbb{R}$ is the orthogonal direct sum of $V_0^\mathbb{R}$ and $T_{v_0}Gv_0$, and therefore $x \in V_0^\mathbb{R}$. Since $v_0 - x \perp T_xX \supseteq T_xX_0$, we find that $x \in X_0$ is a real ED critical point for $v_0$.

\end{proof}
By \ref{hdcrit} and \ref{dense} we may assume that the sufficiently general point on $V^\mathbb{R}$ is in fact a sufficiently general point $v_0$ on $V_0^\mathbb{R}$. The previous two lemmas now tell us that the HD critical points for $v_0$ on $X$ and on $X_0$ are the same. \textbf{Hence the HD degree of $X$ in $V^\mathbb{R}$ is the $\mathbb{R}$ED degree of $X_0$ in $V_0^\mathbb{R}$  }.\\

We now explain how Theorem~\ref{main} fits into the general slicing principle of
Theorem~\ref{dbrevised}.
Throughout we assume $n \le t$ and we consider the left right unitary action
\[
G := U(n)\times U(t) , V := \mathbb{C}^{n\times t},
\qquad (U,V)\cdot A := U A V^{*}.
\]

Define the (real) diagonal slice
\[
V_{0}^{\mathbb{R}}
:= \left\{ \mathrm{diag}(x)\in \mathbb{C}^{n\times t} : x\in \mathbb{R}^{n}\right\},
\]

Let $v_{0}=\mathrm{diag}(y)\in V_{0}^{\mathbb{R}}$ with $y_{1},\dots,y_{n}$ nonzero and
pairwise distinct. Then the tangent space to the $G$ orbit at $v_{0}$ is
\[
T_{v_{0}}(G\cdot v_{0})
=\left\{ Z_{1}\mathrm{diag}(y)+\mathrm{diag}(y)Z_{2}^{*} \;:\;
Z_{1}\in \mathfrak{u}(n),\ Z_{2}\in \mathfrak{u}(t)\right\}.
\]
Moreover $V^{\mathbb{R}}$ splits $q$ orthogonally as
\begin{equation}\label{eq:slice-splitting}
V^{\mathbb{R}} \;=\; V_{0}^{\mathbb{R}} \ \oplus_{\mathbb{R}}\ T_{v_{0}}(G\cdot v_{0}).
\end{equation}
Indeed, for any $a\in \mathbb{R}^{n}$ and any $Z_{1}\in \mathfrak{u}(n)$,
\[
q\!\left(\mathrm{diag}(a),\, Z_{1}\mathrm{diag}(y)\right)
= \mathrm{Re}\,\mathrm{Tr}\!\left(\mathrm{diag}(a)\big(Z_{1}\mathrm{diag}(y)\big)^{*}\right)
= \mathrm{Re}\,\mathrm{Tr}\!\left(\mathrm{diag}(a)\mathrm{diag}(y)Z_{1}^{*}\right)=0,
\]
since $\mathrm{diag}(a)\mathrm{diag}(y)$ is Hermitian and $Z_{1}^{*}$ is skew Hermitian, hence
the trace is purely imaginary. The same argument shows
$q(\mathrm{diag}(a),\, \mathrm{diag}(y)Z_{2}^{*})=0$ for all $Z_{2}\in \mathfrak{u}(t)$.
Therefore
\[
V_{0}^{\mathbb{R}} \subseteq \bigl(T_{v_{0}}(G\cdot v_{0})\bigr)^{\perp_q}.
\]

To prove \eqref{eq:slice-splitting}, it remains to show the reverse inclusion.
Let $A\in V^{\mathbb{R}}$ satisfy $q(A,\,T_{v_0}(G\cdot v_0))=0$.  Since
\[
T_{v_{0}}(G\cdot v_{0})
=\{\, Z_{1}\mathrm{diag}(y) + \mathrm{diag}(y)Z_{2}^{*}
:\ Z_{1}\in\mathfrak{u}(n),\ Z_{2}\in\mathfrak{u}(t)\,\},
\]
we obtain for all $Z_1\in\mathfrak{u}(n)$ and $Z_2\in\mathfrak{u}(t)$ the two conditions
\[
0 = q\!\left(A,\,Z_1\mathrm{diag}(y)\right)
= \mathrm{Re}\,\mathrm{Tr}\!\left(A\,\mathrm{diag}(y)^{*} Z_1^{*}\right),
\qquad
0 = q\!\left(A,\,\mathrm{diag}(y)Z_2^{*}\right)
= \mathrm{Re}\,\mathrm{Tr}\!\left(\mathrm{diag}(y)^{*}A\, Z_2\right).
\]
By Lemma~\ref{skew}(1), the first family implies that $A\,\mathrm{diag}(y)^{*}$ is Hermitian,
and the second implies that $\mathrm{diag}(y)^{*}A$ is Hermitian.  Since $v_0=\mathrm{diag}(y)$
has nonzero diagonal entries with pairwise distinct $|y_i|^{2}$, Lemma~\ref{RD} yields that
$A$ must be a real diagonal matrix (in particular $A\in V_{0}^{\mathbb{R}}$). Hence
\[
\bigl(T_{v_{0}}(G\cdot v_{0})\bigr)^{\perp_q} \subseteq V_{0}^{\mathbb{R}}.
\]
Therefore $\bigl(T_{v_{0}}(G\cdot v_{0})\bigr)^{\perp_q}=V_{0}^{\mathbb{R}}$, and the
$q$ orthogonal direct sum decomposition \eqref{eq:slice-splitting} follows.

Now let $M\subseteq \mathbb{C}^{n\times t}$ be a $G$ stable real algebraic variety (i.e.\ a
unitary invariant matrix variety), and define
\[
S:=\sigma(M)=\{x\in \mathbb{R}^{n}:\ \mathrm{diag}(x)\in M\}.
\]
Set $X:=M$ and $X_{0}:=X\cap V_{0}^{\mathbb{R}}$. By construction we have
\[
X_{0}=\{\mathrm{diag}(x): x\in S\},
\]
so under the identification $V_{0}^{\mathbb{R}}\cong \mathbb{R}^{n}$ given by
$x\leftrightarrow \mathrm{diag}(x)$, the slice $X_{0}$ corresponds exactly to $S$.

Finally, since $M$ is unitary invariant, every $A\in M$ admits an SVD
$A=U\,\mathrm{diag}(\sigma(A))\,V^{*}$ with $U\in U(n)$, $V\in U(t)$, and
$\mathrm{diag}(\sigma(A))\in M$. Hence
\[
M = G\cdot X_{0},
\]
so $G\cdot X_{0}$ is dense in $X=M$. Therefore all hypotheses of
Theorem~\ref{dbrevised} are satisfied (for generic $v_{0}\in V_{0}^{\mathbb{R}}$), and we obtain
\[
\mathrm{HDdeg}_{V}(M)=\mathrm{\mathbb{R}EDdeg}_{V_{0}^{\mathbb{R}}}(X_{0}).
\]
\begin{remark}
This is the same equality as in Theorem~\ref{main}, since
$\mathrm{\mathbb{R}EDdeg}_{V_{0}^{\mathbb{R}}}(X_{0})$ coincides with the real Euclidean distance degree of $S$.
Indeed, on $V_{0}^{\mathbb{R}}$ (the space of real diagonal matrices), the form $q$ restricts to the standard Euclidean inner product:
for $A,B\in V_{0}^{\mathbb{R}}$ we have
\[
\mathrm{Re}\!\big(\mathrm{Tr}(AB^{*})\big)=\mathrm{Tr}(AB^{t}),
\]
so the associated squared norm is $A\mapsto \mathrm{Tr}(AA^{t})$.
\end{remark}
We also show how Theorem \ref{dbrevised} can be used to study Critical points of rank one Hermitian matrices in the next example

Let $G = U(n)$ act on $V$ by conjugation,
\[
g \cdot A = g A g^{\ast}.
\]
This action is unitary since
$\mathrm{Tr}(gAg^{\ast}(gBg^{\ast})^{\ast}) = \mathrm{Tr}(AB^{\ast})$,
so the real part of $q$ defines a $G$ invariant Euclidean structure on $V^{\mathbb{R}}$. 

Define 
\[
V_{0} = \{ \mathrm{diag}(\lambda_1,\dots,\lambda_n) : \lambda_i \in \mathbb{C} \},
\]
the space of diagonal complex matrices.  
For a sufficiently general element
$v_{0} = \mathrm{diag}(\lambda_1,\dots,\lambda_n) \in V_{0}$ with $\lambda_i \neq \lambda_j$ 
for $i \neq j$, the tangent space to the $G$ orbit is
\[
T_{v_0}(G v_0)
= \{ [H,v_0] : H \in \mathfrak{u}(n) \}
= \{ A \in \mathbb{C}^{n\times n} : A \text{ is off diagonal  } \}.
\]
Since diagonal and off diagonal matrices are orthogonal with respect to 
$\mathrm{Re}  (q)$, we have an orthogonal direct sum decomposition
\[
V^{\R} = V_0^\mathbb{R}  \;\oplus _\mathbb{R}\; T_{v_0}(G v_0).
\]

Let $X \subset V$ be the variety of rank one Hermitian matrices,

This is a closed $G$ stable subvariety of $V$.  Its intersection with the slice
$V_0$ is
\[
X_0 = X \cap V_0^{\R}
= \{ \mathrm{diag}(0,\dots,0,\lambda,0,\dots,0) : \lambda \in \mathbb{R} \},
\]
the union of the coordinate axes in $V_0$.  

Every rank one Hermitian matrix is unitarily diagonalizable, hence for each
$X = vv^{\ast} \in X$ there exists $g \in U(n)$ such that 
$gXg^{\ast} \in X_0$.  
Therefore,
\[
G \cdot X_0 = X,
\]
so $G X_0$ is dense in $X$.

All hypotheses of the  Theorem are thus satisfied, and we conclude that
\[
\mathrm{HDdeg}_{V}(X) = \mathrm{\mathbb{R}EDdeg}_{V_0^\mathbb{R}}(X_0).
\]

\newpage
\bibliographystyle{amsalpha}
\bibliography{references}

\end{document}